\documentclass[10pt,draft,reqno]{amsart}
     \makeatletter
     \def\section{\@startsection{section}{1}%
     \z@{.7\linespacing\@plus\linespacing}{.5\linespacing}%
     {\bfseries%\normalfont\scshape
     \centering
     }}
     \def\@secnumfont{\bfseries}
     \makeatother
\usepackage{amsmath, amssymb, bbm}

\usepackage{enumitem}
\newtheorem{theorem}{Theorem}[section]
\newtheorem{lemma}[theorem]{Lemma}
\newtheorem{prop}[theorem]{Proposition}
\newtheorem{cor}[theorem]{Corollary}

\theoremstyle{definition}

\theoremstyle{remark}

\numberwithin{equation}{section}

\newcommand{\rr}{{\mathbb R}}
\newcommand{\rd}{{\mathbb R^d}}
\newcommand{\rmm}{{\mathbb{R}^m}}

\newcommand{\Gr}{\operatorname{Gr}}

\allowdisplaybreaks

\begin{document}
\sloppy
\title[Fractal behavior of operator-self-similar stable random fields]{Fractal behavior of multivariate operator-self-similar stable random fields} 

\author{Ercan S\"onmez}
\address{Ercan S\"onmez, Mathematisches Institut, Heinrich-Heine-Universit\"at D\"usseldorf, Universit\"atsstr. 1, D-40225 D\"usseldorf, Germany}
\email{ercan.soenmez\@@{}hhu.de}

\begin{abstract}
We investigate the sample path regularity of multivariate operator-self-similar stable random fields with values in $\rmm$ given by a harmonizable representation. Such fields were introduced in \cite{LiXiao} as a generalization of both operator-self-similar stochastic processes and operator scaling random fields and satisfy the scaling property $\{X(c^E t) : t \in \rd \} \stackrel{\rm d}{=} \{c^D X(t) : t \in \rd \}$, where $E$ is a real $d \times d$ matrix and $D$ is a real $m \times m$ matrix. This paper provides the first results concerning sample path properties of such fields, including both $E$ and $D$ different from identity matrices. In particular, this solves an open problem in \cite{LiXiao}.
\end{abstract}

\thanks{* This work has been supported by Deutsche Forschungsgemeinschaft (DFG) under grant KE1741/ 6-1}
\maketitle

\baselineskip=18pt

\section{Introduction}

A multivariate operator-self-similar field $\{X (x) : x \in \rd \}$ is a random field with values in $\rmm$ whose finite-dimensional distributions are invariant under suitable scaling of the time vector $x$ and the corresponding $X(x)$ in the state space. More precisely, let $E \in \mathbb{R}^{d \times d}$ and $D \in \mathbb{R}^{m \times m}$ be real matrices with positive real parts of their eigenvalues. Then the random field $\{X(x) : x \in \rd \}$ is called $(E,D)$-operator-self-similar if
\begin{equation}\label{OSS}
	\{ X( c^Ex) : x \in \mathbb{R}^d \} \stackrel{\rm d}{=} \{ c^D X(x) : x \in \mathbb{R}^d \}  \quad \text{for all } c>0,
\end{equation}
where $\stackrel{\rm d}{=}$ means equality of all finite-dimensional  marginal distributions and $c^A = \exp (A \log c) = \sum_{k=0}^{\infty} \frac{(\log c)^k}{k!} A^k$ is the matrix exponential. 

Random fields satisfying the property \eqref{OSS} were first introduced in \cite{LiXiao} as a generalization of both operator-self-similar processes \cite{LahaRo, HudsonMason, Sato, Lamperti} and operator scaling random fields \cite{BMS, BL}. Recall that a stochastic process $\{Z(t) : t \in \rr \}$ with values in $\rmm$ is called operator-self-similar if
\begin{equation*}
	\{ Z( c t) : t \in \mathbb{R} \} \stackrel{\rm d}{=} \{ c^D Z(t) : t \in \mathbb{R} \}  \quad \text{for all } c>0,
\end{equation*}
whereas a scalar valued random field $\{ Y( t) : t \in \mathbb{R}^d \}$ is said to be operator scaling of order $E$ and some $H>0$ if
\begin{equation*}
	\{ Y( c^E t) : t \in \mathbb{R}^d \} \stackrel{\rm d}{=} \{ c^H Y(t) : t \in \mathbb{R}^d \}  \quad \text{for all } c>0.
\end{equation*}
Note that $(E,D)$-operator-self-similar random fields can be seen as an anisotropic generalization of an operator-self-similar random field $\{ Z( t) : t \in \mathbb{R}^d \}$ satisfying
\begin{equation*}
	\{ Z( c t) : t \in \mathbb{R}^d \} \stackrel{\rm d}{=} \{ c^D Z(t) : t \in \mathbb{R}^d \} 
\end{equation*}
for every $c>0$. Then $\{ Z( t) : t \in \mathbb{R}^d \}$ is $(I_d, D)$-operator-self-similar, where $I_d$ is the $d \times d$ identity matrix.

The theoretical importance of self-similar random fields has increased significantly during the past four decades. They are also useful to model various natural phenomena for instance in physics, geophysics, mathematical engineering, finance or internet traffic, see, e.g., \cite{Levy, Abry, SamorodTaqq, Wackernagel, Chiles, BonEstr, Benson, Davis, WillPaxTaqq}. A very important class of such fields is given by Gaussian random fields and, in particular, by fractional Brownian fields (see \cite{SamorodTaqq, MasonXiao}). However, Gaussian modeling is a serious drawback for applications including heavy-tailed persistent phenomena. For this purpose $\alpha$-stable random fields have been introduced. A vector valued random field $\{X (x) : x \in \rd \}$ is said to be symmetric $\alpha$-stable ($S \alpha S$) for $\alpha \in (0,2]$ if any linear combination $\sum_{k=1}^n a_k X(x_k)$ is multivariate $S \alpha S$. We refer the reader to \cite[Chapter 2]{SamorodTaqq} for a comprehensive introduction to multivariate stable distributions.

In order to establish the existence of multivariate operator-self-similar random fields, Li and Xiao \cite{LiXiao} defined stochastic integral representations of random vectors and followed the outline in \cite{BMS}. Both moving-average as well as harmonizable representations of $(E,D)$-operator-self-similar $S\alpha S$ random fields are given. Lastly, they leave the open problem of investigating the sample path regularity and fractal dimensions of these fields. In particular, they conjecture that these properties such as path continuity and Hausdorff dimensions are mostly determined by the real parts of the eigenvalues of $E$ and $D$. S\"onmez \cite{Soenmez} solved this problem for the moving-average and harmonizable representation in the Gaussian case $\alpha = 2$ and generalized several results in the literature (see \cite{BMS, MasonXiao, LiWangXiao, Xiao3}). In particular, he highlighted that the Hausdorff dimension of the range and the graph over a sample path depends on the real parts of the eigenvalues of $E$ and $D$ as well as the multiplicity of the eigenvalues of $E$ and $D$. The purpose of this paper is to establish the corresponding results in the stable case $\alpha \in (0,2)$ for the harmonizable representation. Indeed, we show that harmonizable $\alpha$-stable operator-self-similar random fields have the same kind of regularity properties as Gaussian operator-self-similar random fields. We first give an upper bound on the modulus of continuity by elegantly showing the applicability of results from \cite{BL2} to multivariate operator-self-similar stable random fields. Based on this we calculate the Hausdorff dimension of the range and the graph. We remark that Xiao \cite{Xiao3} investigated Hausdorff dimensions of multivariate $\alpha$-stable random fields by making the assumption of "locally approximately independently components" (see \cite[Section 3]{Xiao3}). In fact, in view of our methods it will be clear that this assumption is superfluous in order to determine the Hausdorff dimension of the range and the graph of the sample paths of multivariate $\alpha$-stable random fields.

The rest of this paper is organized as follows. Section 2 deals with exponential powers of linear operators. In Section 3 we recall the definition of harmonizable multivariate $(E,D)$-operator-self-similar stable random fields. In Section 4 we give an upper bound on the modulus of continuity of the random vector components in terms of the radial part with respect to the matrix $E$ introduced in \cite[Chapter 6]{MeersScheff}. Finally, in Section 5 we state and prove our main results on the Hausdorff dimension of harmonizable multivariate $(E,D)$-operator-self-similar $\alpha$-stable random fields.

\section{Preliminaries}

Throughout this paper, let $E \in \mathbb{R}^{d \times d}$ be a matrix with distinct real parts of its eigenvalues given by $0<a_1 < \ldots < a_p$ for some $p \leq d$ and let $q = \operatorname{trace}(E)$. Assume that each eigenvalue corresponding to $a_1, \ldots, a_p$  has multiplicity $\mu_1, \ldots, \mu_p$, respectively. Furthermore, let $D \in \mathbb{R}^{m \times m}$ be a matrix with positive real parts of its eigenvalues given by $0 < \lambda_1 \leq \lambda_2 \leq \ldots \leq \lambda_m$. As done in \cite{Soenmez} without loss of generality we assume that

\begin{equation}\label{eigenvalues}
	\lambda_m < 1 < a_1.
\end{equation}
Let us recall that from the Jordan decomposition theorem (see e.g. \cite[p. 129]{Hirsch}) there exists a real invertible matrix $A \in \mathbb{R}^{m \times m}$ such that $A^{-1} D A$ is of the real canonical form, i.e. it consists of diagonal blocks which are either Jordan cell matrices of the form
\begin{equation*}
\left(
\begin{array}{cccccc}
\lambda & 1 &  & &    \\
 & \lambda & 1 &   \\
 &  &  \ddots & \ddots&   \\
 & & & \ddots & 1   \\
 & & & &  \lambda   
\end{array} \right) .
\end{equation*}
with $\lambda$ a real eigenvalue of $D$ or blocks of the form
\begin{equation*}
\left(
\begin{array}{cccccc}
\Lambda & I_2 &  & &    \\
 & \Lambda & I_2 &   \\
 &  &  \ddots & \ddots&   \\
 & & & \ddots & I_2   \\
 & & & &  \Lambda   
\end{array} \right) \quad \text{with} \\\  \Lambda = \left(
\begin{array}{cccccc}
a & -b   \\
b & a  
\end{array} \right) \quad \text{and} \\\ I_2 = \left(
\begin{array}{cccccc}
1 & 0   \\
0 & 1  
\end{array} \right),
\end{equation*}
where the complex numbers $a \pm ib, b \neq 0$, are complex conjugated eigenvalues of $D$. The following proposition is due to \cite[Proposition 2.2.11]{MeersScheff}.

\begin{prop} \label{auxprp}
Let $A \in \mathbb{R}^{m \times m}$ be a matrix with positive real parts of its eigenvalues and let $\| \cdot \|$ be any arbitrary norm on $\rmm$. Then the following statements hold.
\item[(a)] If every eigenvalue of $A$ has real part less than $\beta_1$, then for any $t_0 >0$ there exists a constant $C>0$ such that $\| t^A x \| \geq C t^{\beta_1}\| x \|$ holds for all $0<t\leq t_0$ and all $x \in \rmm$.
\item[(b)] If every eigenvalue of $A$ has real part less than $\beta_2$, then for any $s_0 >0$ there exists a constant $C>0$ such that $\| s^A x \| \leq C s^{\beta_2}\| x \|$ holds for all $s \geq s_0$ and all $x \in \rmm$.
\end{prop}

\begin{cor} \label{auxcor}
Assume that $D$ is of the real canonical form and let $\| \cdot \|$ be any arbitrary norm on $\rmm$. Then the following statements hold.
\item[(a)] For any $t_0 >0$ there exists a constant $C_1>0$ such that for any $\varepsilon >0$ 
$$\| t^D \theta \| \geq C_1 \sum_{j=1}^m t^{\lambda_j + \varepsilon} | \theta_j |$$
holds for all $0<t\leq t_0$ and all $\theta \in \rmm$.
\item[(b)] For any $s_0 >0$ there exists a constant $C_2>0$ such that for any $\varepsilon >0$ 
$$\| s^{-D} \theta \| \leq C_2 \sum_{j=1}^m s^{-\lambda_j + \varepsilon} | \theta_j |$$
holds for all $s\geq s_0$ and all $\theta \in \rmm$.
\end{cor}

\begin{proof}
We only prove part (a). Part (b) is left to the reader. Throughout this proof let $c$ be an unspecified positive constant which might change in each occurence. Assume that the distinct real parts of the eigenvalues of $D$ are given by $\overline{\lambda}_1, \ldots,  \overline{\lambda}_k$ for some $k \leq m$ and let us write
\begin{equation*}
\left(
\begin{array}{cccccc}
J_1 &  &  &     \\
 & J_2 &     \\
 &  &  \ddots &   \\
 & & & J_k    \\
\end{array} \right) ,
\end{equation*}
for some block matrices $J_j$ so that each $J_j$ is associated with $\overline{\lambda}_j$, $1 \leq j \leq k$. Furthermore, write $\theta = (\overline{\theta}_1, \ldots, \overline{\theta}_k)$ for any $\theta = (\theta_1, \ldots, \theta_m) \in \rmm$ and let $\| \theta \|_1 = \sum_{j=1}^m |\theta_j|$ be the $1$-norm on $\rmm$. Then, by Proposition \ref{auxprp}, for all $\varepsilon>0$, $t_0>0$ and all $0 < t \leq t_0$ we have
\begin{align*}
\| t^D \theta \| & \geq c  \| t^{D} \theta \|_1 = c \sum_{j=1}^k \| t^{J_j}  \overline{\theta}_j \|_1 \\
& \geq c \sum_{j=1}^k t^{\overline{\lambda}_j + \varepsilon} \| \overline{\theta}_j \|_1 = c \sum_{j=1}^m t^{\lambda_j + \varepsilon} | \theta_j | ,
\end{align*}
where we used the equivalence of norms in the first inequality.
\end{proof}

\section{Harmonizable representation}

Harmonizable stable random vector fields are defined as stochastic vector integrals of deterministic matrix kernels with respect to a stable random vector measure. More precisely, let $\alpha \in (0,2]$, $W_\alpha (du)$ be a $\mathbb{C}^m$-valued isotropic $\alpha$-stable random measure on $\rd$ with Lebesgue control measure (see \cite[Definition 2.1]{LiXiao}) and let $Q(u) = Q_1 (u) + i Q_2 (u)$, where $\{Q_1 (u) : u \in \rd \}$ and $\{Q_2 (u) : u \in \rd \}$ are two families of real $m \times m$ matrices. Let us recall (see \cite[Theorem 2.4]{LiXiao}) that the stochastic integral
$$W_\alpha (Q) := \operatorname{Re} \int_{\rd} Q(u) W_\alpha (du)$$
is well-defined if and only if
$$ \int_{\rd} \big( \|Q_1(u) \|_m^\alpha + \|Q_2(u) \|_m^\alpha \big) du < \infty,$$
where $\|A \|_m = \max_{\|u\| = 1} \|Au\|$ is the operator norm for any matrix $A \in \mathbb{R}^{m \times m}$. Furthermore, in the latter case $W_\alpha(Q)$ is a stable $\mathbb{R}^m$-valued random variable with characteristic function given by
\begin{equation} \label{characfct}
\quad \mathbb{E} \big[ \exp \big(i \langle \theta, W_\alpha(Q) \rangle \big) \big] = \exp \Big( -  \int_{\rd} \big( \sqrt{ \|Q_1(u) \theta \|^2 + \|Q_2(u) \theta \|^2  }\big) ^\alpha du \Big) .
\end{equation}
for all $\theta \in \rmm$. Note that $W_2(Q)$ is a centered Gaussian random vector. Let $\psi : \rd \to [0, \infty )$ be a continuous $E^T$-homogeneous function, which means according to \cite[Definition 2.6]{BMS} that
\begin{equation*}
	\psi (c^{E^T} x) = c \psi (x)  \quad \text{for all } c>0.
\end{equation*}
Moreover assume that $\psi (x) \neq 0$ for $x \neq 0$ and let $I_m$ be the identity operator on $\rmm$. Recall that $q= \operatorname{trace}(E)$. Li and Xiao \cite{LiXiao} proved the following.

\begin{theorem}\label{harmonizable}
If \eqref{eigenvalues} is fullfilled, the random field
\begin{equation}\label{harmon}
X_{\alpha} (x) = \operatorname{Re} \int_{\rd} ( e^{i \langle x, y \rangle} -1 ) \psi (y) ^{-D - \frac{q I_m}{\alpha}} W_{\alpha} (d y) , \quad x \in \rd
\end{equation}
is well defined and called harmonizable $(E,D)$-operator-self-similar random field.
\end{theorem}
From Theorem 2.6 in \cite{LiXiao} $X_\alpha$ is a proper, stochastically continuous random field with stationary increments and satisfies the scaling property \eqref{OSS}. Let us recall that an $\rmm$-valued random field $\{ Y(t) : t \in \rd\}$ is said to be proper if for every $t \in \rd$ the distribution of $Y(t)$ is full, i.e. it is not supported on any proper hyperplane in $\rmm$.

As noted above, S\"onmez \cite{Soenmez} studied the sample path properties of $X_\alpha$ in the Gaussian case $\alpha =2$. We will derive similar results in this paper for $X_\alpha$ with $\alpha \in (0,2)$. We first give an upper bound on the modulus of continuity of the components in the next Section.

\section{Modulus of continuity}

Throughout this Section assume that $\alpha \in (0,2)$. For notational convenience let us surpress the subscript $\alpha$ and simply write $X$ instead of $X_\alpha$. Furthermore, let $\tau_E(\cdot )$ be the radial part of polar coordinates with respect to $E$ introduced in \cite[Chapter 6]{MeersScheff} (see also \cite{BMS, BL, BL2, LiXiao, Soenmez}). The following is the main result of this Section.

\begin{prop} \label{modulusofcontinuity}
Assume that the operator $D$ is of the real canonical form. Then, there exists a modification $X^*$ of $X$ such that for any $\varepsilon >0$ and any $\delta >0$
\begin{align}\label{modulus}
 \sup_{\substack{u,v \in [0,1]^d\\ u \neq v}} \frac{|X^*_j(u)-X^*_j(v)|}{\tau_E(u-v)^{\lambda_j - \varepsilon} \big[ \log \big( 1 + \tau_E(u-v)^{-1} \big) \big]^{\delta+\frac{1}{2} + \frac{1}{\alpha}} } < \infty 
\end{align}
holds almost surely for all $j=1, \ldots, m$. In particular, for every $\varepsilon >0$ and $j=1, \ldots, m$, there exists a constant $C_{4,1}>0$ such that $X^*$ satisfies a.s.
\begin{align}\label{continuity}
|X^*_j(u)-X^*_j(v)| \leq C_{4,1} \tau_E(u-v)^{\lambda_j-\varepsilon} \quad \text{for all} \quad u, v \in [0,1]^d.
\end{align}
\end{prop}

In order to prove Proposition \ref{modulusofcontinuity} we recall a result that has recently been established by Bierm\'e and Lacaux \cite{BL2}. The key point is to remark that the components $X_j$, $1 \leq j \leq m$, behave like one-dimensional operator scaling harmonizable random fields given in \cite{BMS}. Let $M_\alpha (d\xi)$ be a complex isotropic $\alpha$-stable random measure on $\rd$ with Lebesgue control measure as introduced in \cite[p. 281]{SamorodTaqq}. Furthermore, let $Y$ be a scalar valued random field defined through the stochastic integral
\begin{equation} \label {auxrf}
Y = \Big( \operatorname{Re} \int_{\rd} f_\alpha (u, \xi ) M_\alpha (d\xi) \Big) _{u \in \rd } ,
\end{equation}
where $f_\alpha (u, \cdot ) \in L^\alpha ( \rd )$ is given by
\begin{align*}
f_\alpha (u, \xi ) = ( e^{i \langle u, \xi \rangle} -1 ) \psi_\alpha (\xi) \quad \forall  (u, \xi) \in \rd \times \rd ,
\end{align*}
with a Borel measureable function $\psi_\alpha : \rd \to \mathbb{C}$ satisfying
$$ \int_{\rd} \min (1, \| \xi \|^\alpha ) |\psi_\alpha (\xi ) |^\alpha d\xi < \infty .$$

Then, Bierm\'e and Lacaux \cite{BL2} proved the following.

\begin{lemma} \label{helplemma}
Assume that there exist some positive and finite constants $c_\psi, K$ and $\beta \in (0, a_1)$ such that
$$|\psi_\alpha ( \xi ) | \leq c_\psi \tau_{E^T} ( \xi )^{- \beta - \frac{q}{\alpha}} $$
holds for almost every $\xi \in \rd$ with $\| \xi \| > K$. Then, there exists a modification $Y^*$ of $Y$ such that almost surely for every $\delta >0$
\begin{align}\label{auxmodulus}
 \sup_{\substack{u,v \in [0,1]^d\\ u \neq v}} \frac{|Y^*(u)-Y^*(v)|}{\tau_E(u-v)^{\beta} \big[ \log \big( 1 + \tau_E(u-v)^{-1} \big) \big]^{\delta+\frac{1}{2} + \frac{1}{\alpha}} } < \infty .
\end{align}
\end{lemma}

\noindent \textit{Proof of Proposition \ref{modulusofcontinuity}.}
Fix $1 \leq j \leq m$ and denote by $(e_1, \ldots, e_m)$ the canonical basis of $\rmm$. The main idea is to apply Lemma \ref{helplemma} with an appropriate choice of the function $\psi_\alpha$. Indeed, let $Y$ be the random field given in \eqref{auxrf} with
$$\psi_\alpha (\xi) = \| \psi ( \xi ) ^{-D - \frac{q}{\alpha} I_m} e_j \| ,$$
where $\| \cdot \|$ is any arbitrary norm on $\rmm$. Using \eqref{characfct}, it is easy to see that, up to a multiplicative constant,
$$\{X_j(u) : u \in \rd \} \stackrel{\rm d}{=} \{ Y(u) : u \in \rd \} .$$
Since $\psi$ is $E^T$-homogeneous, from Corollary \ref{auxcor} one easily checks that
$$ \| \psi (\xi )  ^{-D - \frac{q}{\alpha} I_m} e_j \|  \leq c_\psi \tau_{E^T} ( \xi ) ^{- (\lambda_j - \varepsilon) - \frac{q}{\alpha}} $$
for all $\xi \in \rd$ with $\| \xi \| > K$ and some $c_\psi \in (0 , \infty )$. Therefore, by Lemma \ref{helplemma}, there exists a modification $Y^*$ of $Y$ such that $Y^*$ almost surely satisfies \eqref{auxmodulus} with $\beta = \lambda_j - \varepsilon$. Further note that, since $Y^*$ is a modification of $Y$, we have
$$\{X_j(u) : u \in \rd \} \stackrel{\rm d}{=} \{ Y^*(u) : u \in \rd \} ,$$
so that $X_j$ almost surely satisfies \eqref{modulus} for countably many $u,v \in [0,1]^d$. Using this and the fact that $X_j$ is stochastically continuous, exactly as in the proof of \cite[Proposition 5.1]{BL2} one can define a modification $X_j^*$ of $X_j$ such that \eqref{modulus} holds. We omit the details.
\hfill $\Box$
\newline

Proposition \ref{modulusofcontinuity} compared to \cite[Proposition 4.6]{Soenmez} shows that $(E,D)$-operator-self-similar stable random fields share the same kind of upper bound for the modulus of continuity as the Gaussian ones. Therefore it is natural to have also the same results of \cite[Theorem 4.1]{Soenmez} for the Hausdorff dimension of their images and graphs on $[0,1]^d$, which we state in the next Section. Furthermore, we  refer the reader to \cite{Fal, Mattila} for the definition and properties of the Hausdorff dimension.

\section{Hausdorff dimension of the image and the graph}

The main result of this Section is the following.

\begin{theorem}\label{hausdorff}
Let the assumptions of the previous Sections hold and let $\alpha \in (0,2)$. Then, almost surely
\begin{align}\label{imdim}
	\dim_{\mathcal{H}} X ([0,1]^d) &= \min \Big\{ m, \frac{ \sum_{k=1}^p a_k \mu_k + \sum_{i=1}^j (\lambda_j - \lambda_i) }{ \lambda_j }, 1 \leq j \leq m \Big\}
\end{align}

\begin{align}\label{imdim2}
	 &= \begin{cases}		
				 m & \mbox{if } \sum_{i=1}^m \lambda_i < \sum_{k=1}^p a_k \mu_k ,\\ 				 \frac{ \sum_{k=1}^p a_k \mu_k + \sum_{i=1}^l (\lambda_l - \lambda_i) }{ \lambda_l }  & \mbox{if }  \sum_{i=1}^{l-1} \lambda_i < \sum_{k=1}^p a_k \mu_k \leq \sum_{i=1}^{l} \lambda_i ,		
		\end{cases}
\end{align}

\begin{align}\label{graphdim}
	\dim_{\mathcal{H}} \Gr X ([0,1]^d) &= \min \Bigg\{ \frac{ \sum_{k=1}^p a_k \mu_k + \sum_{i=1}^j (\lambda_j - \lambda_i) }{ \lambda_j }, 1 \leq j \leq m, \\ 
	& \quad \sum_{j=1}^l \frac{\tilde{a}_j}{\tilde{a}_l} \tilde{\mu}_j + \sum_{j=l+1}^p  \tilde{\mu}_j + \sum_{i=1}^m (1- \frac{\lambda_i}{\tilde{a}_l}) , 1 \leq l \leq p \Bigg\} \nonumber
\end{align}
\begin{footnotesize}
\begin{align}\label{graphdim2}
	 &= \begin{cases} 
			\begin{aligned}
				& \dim_{\mathcal{H}} X ([0,1]^d) \,\,\,\quad\qquad\qquad\qquad\,\,\, \mbox{ if } \sum_{k=1}^p a_k \mu_k \leq \sum_{i=1}^m \lambda_i  ,\\ 		
		&   \sum_{j=1}^l \frac{\tilde{a}_j}{\tilde{a}_l}  \tilde{\mu}_j + \sum_{j=l+1}^p  \tilde{\mu}_j + \sum_{i=1}^m (1- \frac{\lambda_i}{\tilde{a}_l}) \mbox{ if } \sum_{k=1}^{l-1} \tilde{a}_k  \tilde{\mu}_k \leq \sum_{i=1}^m \lambda_i < \sum_{k=1}^{l} \tilde{a}_k  \tilde{\mu}_k ,
			\end{aligned}
		\end{cases}
\end{align}
\end{footnotesize}
where $\tilde{a}_j = a_{p+1-j}, \tilde{\mu}_j = \mu_{p+1-j}, 1 \leq j \leq p$ are defined as in Section 2.
\end{theorem}

Before proving Theorem \ref{hausdorff}, we first prove the following.

\begin{lemma} \label{cfbound}
Assume that $D$ is of the real canonical form. Then for all $t \in \rd, \theta \in \rmm$ and $\varepsilon >0$ there exists a constant $C_{5,1}>0$, depending on $\varepsilon$ only, such that
$$ \mathbb{E} \big[ \exp \big( i \langle X(t), \theta \rangle \big) \big] \leq \exp \Big( - C_{5,1} \sum_{j=1}^m \big| \tau_E (t) ^{\lambda_j + \varepsilon}  | \theta_j | \big| ^\alpha \Big) .$$
\end{lemma}

Let us recall that $X(t)$ is defined in \eqref{harmon} with characteristic function given by \eqref{characfct}.
\begin{proof}
Let $\big( \tau_E (t), l_E (t) \big)$ be the polar coordinates of $t$ with respect to $E$ according to \cite[Section 2]{BMS} and $S_E = \{ t \in \rd : \tau_E(t) = 1 \}$. Further, let $c$ be an unspecified positive constant which might change in every occurence. Using the characteristic function of the $S\alpha S$ random vector $X(t)$  and the $E^T$-homogenity of $\psi$, by Corollary \ref{auxcor}, the change to polar coordinates and a substitution we get
\begin{equation} \label{cfineq}
\begin{split}
& \mathbb{E} \big[ \exp \big( i \langle X(t), \theta \rangle \big) \big] \\
& = \exp \Big( - \int_{\rd} | \exp \big( i \langle l_E(t) , y \rangle \big) - 1 |^\alpha \| \tau_E (t)^D \psi (y) ^{-D - \frac{q}{\alpha} I_m} \theta \|^\alpha dy \Big) \\
& \leq  \exp \Big( - c \int_{\rd} | \exp \big( i \langle l_E(t) , y \rangle \big) - 1 |^\alpha \psi (y) ^{-(\lambda_m - \varepsilon) \alpha - q} dy \Big|  \sum_{j=1}^m \tau_E (t) ^{\lambda_j + \varepsilon}  | \theta_j | \Big| ^\alpha \Big) \\
& \leq  \exp \Big( - c m_\alpha \cdot  \big| \tau_E (t) ^{\lambda_j + \varepsilon}  | \theta_j | \big| ^\alpha \Big)
\end{split}
\end{equation}
for all $1 \leq j \leq m,$ with $m_\alpha = \min_{\xi \in S_E} \int_{\rd} | e^{  i \langle\xi , y \rangle  - 1 }|^\alpha \psi (y) ^{-(\lambda_m - \varepsilon) \alpha - q} dy$ a positive and finite constant. Now let $X^{(1)}, \ldots , X^{(m)}$ be independent copies of $X(t)$. Then, since $X(t)$ is an $S\alpha S$ random vector, by \cite[Corollary 2.1.3]{SamorodTaqq}, we have
$$ m^{-\frac{1}{\alpha}} ( X^{(1)} + \ldots + X^{(m)} ) \stackrel{\rm d}{=} X(t) .$$
Using this and \eqref{cfineq}, we get
\begin{align*}
\mathbb{E} \big[ \exp \big( i \langle X(t), \theta \rangle \big) \big] & = \mathbb{E} \big[ \exp \big( i \langle  m^{-\frac{1}{\alpha}} \sum^m_{j=1}  X^{(j)} , \theta \rangle \big) \big]  \\
& = \prod^m_{j=1} \mathbb{E} \big[ \exp \big( i m^{-\frac{1}{\alpha}} \langle  X^{(j)} , \theta \rangle \big) \big]  \\
& \leq  \prod^m_{j=1}  \exp \Big( - c  \big| \tau_E (t) ^{\lambda_j + \varepsilon}  | \theta_j | \big| ^\alpha \Big) \\
& = \exp \Big( - c \sum_{j=1}^m \big| \tau_E (t) ^{\lambda_j + \varepsilon}  | \theta_j | \big| ^\alpha \Big)
\end{align*}
as desired.
\end{proof}

Let us now give a proof of Theorem \ref{hausdorff}.
\newline

\noindent \textit{Proof of Theorem \ref{hausdorff}.}
Let us choose a continuous version of $X$. By using the same argument as in the proof of \cite[Theorem 4.1]{Soenmez}, without loss of generality we will assume that $D$ is of the real canonical form. Then, the upper bounds in Theorem \ref{hausdorff} follow from \eqref{continuity} and Lemma 4.4 in \cite{Soenmez}. So it remains to prove the lower bounds in Theorem \ref{hausdorff}. We will do this by applying Frostman's criterion  (see e.g. \cite{Adler, Fal, Kahane, Mattila}). Throughout this proof, let $c$ be an unspecified positive constant which might change in each occurence. Let us first prove the lower bound in (5.1). One only has to prove that the expected energy integral
\begin{align*}
\mathcal{E}_\gamma = \int_{[0,1]^d \times [0,1]^d} \mathbb{E} [ \|X(x) - X(y) \|^{-\gamma} ] dx dy
\end{align*}
is finite to get that $\dim_{\mathcal{H}} X ([0,1]^d) \geq \gamma$ almost surely. Recall that (see \cite[p. 283]{Xiao3}) for any random vector $Y$ with values in $\rmm$
\begin{align*}
2^{\frac{\gamma}{2}-1} \Gamma (\frac{\gamma}{2}) (2 \pi)^{-\frac{m}{2}} \mathbb{E} \big[ \| Y \|^{-\gamma} \big] = \int_0^\infty \int_{\rmm} \exp \big( - \frac{\|y\|^2}{2} \big) \mathbb{E} \big[ \exp \big( i \langle uy, Y \rangle \big) \big] dy u^{\gamma-1} du.
\end{align*}
Using this and the fact that $X$ has stationary increments, by Lemma \ref{cfbound} for all $\lambda'_j > \lambda_j , 1 \leq j \leq m$, we get that
\begin{align*}
\eta_\gamma & = \mathbb{E} [ \|X(t) - X(s) \|^{-\gamma} ] \leq c \int_\rmm \| x \| ^{\gamma - m} \exp \Big( - c \sum_{j=1}^m \big| \tau_E (t-s) ^{\lambda'_j}  | x_j | \big| ^\alpha \Big) dx \\
& = c \tau_E (t-s) ^{- \sum_{j=1}^m \lambda'_j}  \int_\rmm \Big[ \sum_{j=1}^m \Big( \tau_E (t-s) ^{- \lambda'_j} y_j \Big) ^2  \Big]^ {\frac{\gamma - m}{2}} \exp \Big( - c \sum_{j=1}^m |y_j |^\alpha \Big) dy_1 \ldots dy_m \\
& = c \tau_E (t-s) ^{- \sum_{j=1}^m \lambda'_j + (m - \gamma ) \lambda'_m}  \\
& \quad \times \int_\rmm \Big[ y_m^2 + \sum_{j=1}^{m-1} \Big( \tau_E (t-s) ^{\lambda_m - \lambda'_j} y_j \Big) ^2  \Big]^ {\frac{\gamma - m}{2}} \exp \Big( - c \sum_{j=1}^m |y_j |^\alpha \Big) dy_m \ldots dy_1.
\end{align*}
From the proof of \cite[Theorem 3.1]{Xiao3} we immediately get that
$$\eta_\gamma \leq c \tau_E (t-s) ^{- \sum_{j=1}^{m-k} (\lambda'_{m-k} -  \lambda'_j) - \gamma \lambda'_{m-k}} $$
for all $0 \leq k \leq m-1$ as soon as $m-k-1 < \gamma < m - k$ or, equivalently,
$$\eta_\gamma \leq c \tau_E (t-s) ^{- \sum_{j=1}^{k} (\lambda'_{k} -  \lambda'_j) - \gamma \lambda'_{k}} $$
for all $1 \leq k \leq m$ as soon as $k-1 < \gamma <  k$. Since, for all $1 \leq k \leq m$, the integral
$$\int_{[0,1]^d \times [0,1]^d} \tau_E (t-s) ^{- \sum_{j=1}^{k} (\lambda'_{k} -  \lambda'_j) - \gamma \lambda'_{k}}  dt ds$$
is shown to be finite in \cite{Soenmez} for all
$$ 0 < \gamma < \min \Big\{ m, \frac{ \sum_{j=1}^p a_j \mu_j + \sum_{i=1}^k (\lambda'_k - \lambda'_i) }{ \lambda'_k } \Big\} ,$$
this proves (5.1) by letting $\lambda'_i \to \lambda_i, 1 \leq i \leq m$.

It remains to prove the lower bound in (5.3).  Again, by Frostman's criterion, it suffices to show that
$$\mathcal{G} _\gamma = \int_{[0,1]^d \times [0,1]^d} \mathbb{E} \big[ \big( \|x-y\|^2 + \|X(x) - X(y) \|^2 \big) ^{-\frac{\gamma}{2}} \big] dx dy < \infty$$
in order to obtain $\dim_{\mathcal{H}} \Gr X ([0,1]^d) \geq \gamma$ almost surely. We do this by generalizing the Fourier inversion method used in \cite{BMS, BL}, see also \cite{BCI}. First note that, since $\dim_{\mathcal{H}} \Gr X ([0,1]^d) \geq \dim_{\mathcal{H}} X ([0,1]^d)$ always holds, the case $\dim_{\mathcal{H}} X ([0,1]^d) < m$ and the corresponding upper bound in (5.3) imply that $\dim_{\mathcal{H}} \Gr X ([0,1]^d) = \dim_{\mathcal{H}} X ([0,1]^d)$ almost surely. The dimension of the graph can be larger as the dimension of the range if $\dim_{\mathcal{H}} X ([0,1]^d) = m$, so, in the following, let us assume that $\gamma >m$.  Furthermore, for all $\xi \in \rmm$, let us define a function $f_\gamma (\xi ) = \big( \| \xi \| ^2 + 1 \big) ^{-\frac{\gamma}{2}}$ and denote its Fourier transform by $\hat{f}_\gamma$. Since $\gamma > m$, we have that $\hat{f}_\gamma \in L^\infty (\mathbb{R} ) \cap  L^1 (\mathbb{R} )$, i.e. $\hat{f}_\gamma$ belongs to the set of essentially bounded functions. Using this and Fourier inversion, one easily gets that
$$\zeta_\gamma =  \mathbb{E} \big[ \big( \|s-t\|^2 + \|X(s) - X(t) \|^2 \big) ^{-\frac{\gamma}{2}} \big] \leq c \|s-t\|^{-\gamma} \int_\rmm \mathbb{E} \Big[ e^{i \langle y, \frac{ X(t-s)}{\|t-s\|} \rangle} \Big] dy.$$
From Lemma \ref{cfbound}, we further obtain for all $ \lambda'_j >  \lambda_j, 1 \leq j \leq m$,
\begin{align*}
\zeta_\gamma & \leq c \|s-t\|^{-\gamma} \int_\rmm \exp \Big( -c   \sum_{j=1}^m \big| \tau_E (t-s) ^{\lambda'_j}   \frac{|y_j|}{\|s-t\|}  \big| ^\alpha \Big) dy \\
& = c \|s-t\|^{m-\gamma} \int_\rmm \exp \Big( -c   \sum_{j=1}^m \big| \tau_E (t-s) ^{\lambda'_j}   |x_j|  \big| ^\alpha \Big) dx \\
& = c \|s-t\|^{m-\gamma} \tau_E (t-s) ^{- \sum_{j=1}^m \lambda'_j} .
\end{align*}
The rest of the proof follows from the proof of \cite[Theorem 4.1]{Soenmez}. Finally, (5.2) and (5.4) are easily verified, see \cite[Lemma 4.2]{Soenmez}.
\hfill $\Box$

\bibliographystyle{plain}

\end{document}